\documentclass[12pt]{article}
\usepackage{amsthm,amssymb,amsmath,lineno}
\usepackage{epsfig}
\usepackage{tikz}
\usepackage{adjustbox}
\usepackage{pstricks,pspicture}
\usepackage{graphicx}
\usepackage{textgreek}
\usepackage{color}
\usepackage{pst-plot}
\usepackage{epsfig}
\usepackage{setspace}
\usepackage{float}
\usepackage{rotating}
\usepackage[left=2cm,top=2cm, bottom=2cm,right=2cm]{geometry}

\newtheorem{prelem}{{\bf Proposition}}

\newtheorem{theorem}{Theorem}
\newtheorem{proposition}{Proposition}

\newtheorem{corollary}[theorem]{Corollary}
\newtheorem{observation}[theorem]{Observation}

\theoremstyle{definition}

\theoremstyle{remark}

\title{Revisiting the outer-weakly convex domination number in graph products }
\date{}
\author
{
	Bijo S. Anand$^1$,  Ullas Chandran S. V.$^2$,\\ Jonecis A. Dayap$^3$, Leomarich F. Casinillo$^4$, Karen Luz P. Yap$^4$ \\
	\\
	$^1$ Department of Mathematics, Sree Narayana College, Punalur - 691305, \\Kollam, Kerala, India, email:bijos\_anand@yahoo.com\\$^2$ Department of Mathematics, Mahatma Gandhi College, \\Thiruvananthapuram - 695004, Kerala, India, email:  svuc.math@gmail.com \\$^3$University of San Jose - Recoletos, Philippines, \\email:jdayap@usjr.edu.ph \\ $^4$Visayas State University, Philippines, \\email:leomarichcasinillo02011990@gmail.com, karen.yap@vsu.edu.ph 
}

\begin{document}
\maketitle

\begin{abstract}
Let $G = (V, E)$ be a simple undirected connected graph. A set $C \subseteq V(G)$ is weakly convex in $G$ if for every two vertices $u,v$ in $G$, there exists a $u-v$ geodesic whose vertices are in $C$. A set $C \subseteq V$ is an outer-weakly convex dominating set if every vertex not in $C$ is adjacent to some vertex in $C$ and the set $V(G)\setminus C$ is weakly convex in $G$. The outer-weakly convex domination number of graph $G$, denoted by $\widetilde{ \gamma}_{wcon}(G)$, is the minimum cardinality of an outer-weakly convex dominating set of graph $G$. In this paper, we determine the outer-weakly convex domination number of two graphs under the Cartesian, strong and lexicographic products, and discuss some important combinatorial findings. 

\end{abstract}

\noindent{\bf Keywords}: convex set, dominating set, weakly convex set, Cartesian product, strong product, lexicographic product.       

\noindent {\bf AMS Subj. Clas}: 05C38, 05C76, 05C99, 52A01.

%%%%%%%%%%%%%%%%%%%%%%%%%%%%%%%%%%%%%%%%%%%%%%%%%%%%%%%%%%%%%%%%%%%%
%%%%%%%%%%%%%%%%%%%%%%%%%%%%%%%%%%%%%%%%%%%%%%%%%%%%%%%%%%%%%%%%%%%%%
\section{Introduction}
Domination number is one of the topics in graph theory that remains intriguing. Currently, there are numerous research activities happening in the area and some interesting results can be found in \cite{r1, r2, r3, r5, r6}. In fact, there are more than a hundred of domination parameters published and still rapidly growing \cite{r1, r2, r4, r5, r8, r9}. In that case, this made the essential part to the mathematicians' motivation in conducting a research paper in this particular area. In other words, the authors of this paper have explored a recent parameter called outer-weakly convex domination number of graphs \cite{r7}. Additionally, new results from this parameter were constructed under some product operations in graphs. So, with that exploration, we need some important definitions in graph theory that will be useful throughout this paper and which can be found in \cite{r3, r6}.

Let $G=(V, E)$ be a simple undirected graph. A set $C \subseteq V(G)$ is called a \emph{dominating set} of the graph $G$ if, for every vertex $x \in V(G) \setminus C$, there exists a vertex $y \in C$ such that $xy$ is an edge in $G$. In that case, the \emph{domination number} of $G$, denoted by $\gamma(G)$, is the minimum cardinality of a dominating set $C$ of $G$. 

A graph $G$ is said to be \emph{connected} if there exists at least one path between each pair of vertices $u$ and $v$ in $G$; otherwise, $G$ is disconnected.
For arbitrary vertices $x$ and $y$ in a connected graph $G$, the \emph{distance} between $x$ and $y$, denoted by $d_G(x, y)$, is the length of a shortest path in graph $G$. A shortest $x-y$ path is also known as an $x-y$ \emph{geodesic} and the \emph{closed interval}, denoted by $I_G[x,y]$, consists of all those vertices that are lying on an $x-y$ geodesic in a graph $G$. For a subset $D$ of vertices of a graph $G$, the union of all sets $I_G[x, y]$ for $x, y \in D$ is denoted by $I_G[D]$. Thus, a vertex $u$ is in $I_G[D]$ if and only if $u$ is lying on some $x-y$ geodesic, where $x, y \in D$. A set $D$ is \emph{convex} if and only if $I_G[D]=D$. Consequently, if a graph $G$ is a connected graph, then it follows directly that $V(G)$ is a convex set. For more concepts involving convexity in graphs, the readers are referred to \cite{r8, r9, r10, r_peleyo}.

A subset $S$ of vertices of a graph $G$ is called \emph{outer-convex dominating} if it is a dominating set and the subgraph induced by $V(G)\setminus S$,
denoted by $\langle V(G)\setminus S \rangle $, is convex. The \emph{outer-convex domination number} $\widetilde{ \gamma}_{con}(G)$ is the
minimum cardinality of an outer-convex dominating set $S$ of $G$. Now, a subset $C$ of vertices is called \emph{weakly convex} if for every pair of vertices $x,y\in C$, there is a $x-y$ geodesic whose vertices are in $C$. A subset $D \subseteq V(G)$ is an \emph{outer-weakly convex dominating set} if it is a dominating set and the set $V(G)\setminus D$ is weakly convex. The \emph{outer-weakly convex domination number} of a graph $G$, denoted by $\widetilde{ \gamma}_{wcon}(G)$, is the minimum cardinality of an outer-weakly convex dominating set of $G$. For studies involving outer-convex dominating sets and outer-weakly convex dominating sets, the readers are referred to \cite{r4, r7, r8}. Moreover, we need the following concepts of product operations in graphs that are studied in \cite{r11, r12, r13}.

For the three standard products of a pair of graphs $G$ and $H$, the vertex set of the product is $V(G)\times V(H)$. Their edge sets are defined as follows. In the \emph{Cartesian product} $G \square H$, the vertices $(g, h)$ and $(g', h')$ are adjacent if and only if either 
i) $gg'\in E(G)$ in $G$ and $h = h'$, or 
ii) $g = g'$ and $hh'\in E(G)$ in $H$. In the \emph{strong product}, the vertices $(g, h)$ and $(g', h')$ are adjacent if one of the following holds: 
i) $gg'\in E(G)$ and $h = h'$, 
ii) $g = g'$ and $hh'\in E(G)$, or 
iii) $gg'\in E(G)$ and $hh'\in E(H)$. Finally, two vertices $(g,h)$ and $(g^{\prime },h^{\prime })$ are adjacent in the
\emph{lexicographic product} $G\circ H$ if either $gg^{\prime}\in E(G)$ or if $g=g^{\prime }$ and $hh^{\prime }\in E(H)$. For
$\ast \in \{\Box ,\boxtimes,\circ \}$ we call the product $G\ast H$ \emph{nontrivial}, if both $G$ and $H$ have at least two vertices.
For $h\in V(H)$, $g\in V(G)$, and $\ast \in \{\Box ,\boxtimes ,\circ \}$, call $G^{h}=\{(g,h)\in G\ast H:\,g\in V(G)\}$ a $G$ \emph{layer} in $G\ast H$,
and call $^{g}H=\{(g,h)\in G\ast H:\,h\in V(H)\}$ an $H$ $\emph{layer}$ in $G\ast H$. Note that the subgraph of $G\ast H$ induced on $G^{h}$ is isomorphic to $G$ and the subgraph of $G\ast H$ induced on $^{g}H$ is isomorphic to $H$ for $\ast
\in \{\Box ,\boxtimes ,\circ \}$. Note also that all the three products are associative and only the first two are commutative while the lexicographic product is not. The map
$\pi_{G}:V(G\ast H)\rightarrow V(G)$ defined with $\pi_{G}((g,h))=g$ is called a \emph{projection map onto} $G$ for $\ast \in \{\Box ,\boxtimes,\circ \}$. Similarly, we can define the \emph{projection map onto} $H$. Set finally $[n] = \{ 1, 2, \dots, , n \}$ , where $n \in \mathbb{N}$.
%Let $S\subset V(G)$. With $\left\langle S\right\rangle $ we denote the subgraph of $G$ induced by $S$. 

 In this paper, we explore the concept of outer-weakly convex domination number of two graphs using combinatorial properties under some product operations. Additionally, Some bounds on the outer-weakly convex domination number are also determined and we extract some important  properties that suits the interest of graph theorists.

%\section{Outer-Weakly Convex Domination Number}

\section{Cartesian Product of Graphs}

In this section, we investigate the outer-weakly convex domination number of two graphs under the Cartesian product operation.
The following observation about dominating sets in the Cartesian product is immediate but is used implicitly in many places.
\begin{observation}\label{obs1}
   Let $G$ and $H$ be two non-trivial connected graphs, and let $S_1$ and $S_2$ be proper subsets of $V(G)$ and $V(H)$, respectively. Then, $S_1 \times S_2$ is not a dominating set in $G \Box H$.
\end{observation}
%\begin{conjecture} For any two graphs $G$ and $H$,
%$\gamma(G\Box)H\geq \gamma(G)\gamma(H)$ 
%\end{conjecture}
As mentioned in \cite{r_peleyo}, for any convex set $S$ in $G \Box H$, both projections $\pi_G(S)$ and $\pi_H(S)$ are convex in the corresponding factor graphs. Proposition~\ref{prop_outercon} shows that a similar result holds for outer-convex dominating sets. However, the same type of result need not hold for outer-weakly convex dominating sets, as illustrated in Figure $1$.

\begin{proposition}\label{prop_outercon}
    Let $G$ and $H$ be two non-trivial connected graphs and let $S$ be an outer-convex dominating set in $G\Box H$. Then, both $\pi_G(S)$ and $\pi_H(S)$ are outer-convex dominating sets in $G$ and $H$, respectively.
\end{proposition}
\begin{proof}
    Suppose on the contrary that $\pi_G(S)$ is not an outer-convex dominating set in $G$. Then, either $\pi_G(S)$ is not a dominating set in $G$, or $V(G)\setminus \pi_G(S)$ is not convex in $G$, or both. First consider the case that $\pi_G(S)$ is not a dominating set in $G$. Then, there exists a vertex $g\in V(G)\setminus \pi_G(S)$ with $N_G(g)\cap \pi_G(S)=\emptyset$. Hence, for any $h\in V(H)$, the vertex $(g,h)$ is not adjacent to any vertex of $S$. This contradicts the fact that $S$ is a dominating set in $G$.

    On the other hand, suppose that $V(G)\setminus \pi_G(S)$ is not a convex set in $G$. Then, there exist vertices $g', g'' \in V(G) \setminus \pi_G(S)$ and a shortest $g'$–$g''$ path $P$ in $G$ that intersects $\pi_G(S)$.
 We fix $P: g' = g_1, g_2, \ldots, g_r = g''$ with $g_i \in \pi_G(S)$ for some $i \in [r]$. Now, choose $h \in V(H)$ such that $(g_i, h) \in S$.  
Note that both $(g', h)$ and $(g'', h)$ are not in $S$.
This in turn implies that the path $Q: (g',h)=(g_1,h),(g_2,h), \ldots,(g_k,h)= (g'',h)$ is a $(g',h)-(g'',h)$ shortest path in $G\Box H$ containing the vertex $(g_i, h)$, a contradiction to the fact that $S$ is  outer-convex in $G\Box H$. Hence $\pi_G(S)$ is an outer-convex dominating set in $G$. Similarly, we can prove that $\pi_H(S)$ is an outer-convex dominating set in $H$.
\end{proof}

%\begin{proof}
    %Since $S_1$ and $S_2$ are two proper subsets of $V(G)$ and $V(H)$, respectively, we can find vertices $g\in V(G)\setminus S_1$ and $h\in V(H)\setminus S_2$. Then, by the definition of Cartesian product of graph, $V(^gH)\nsubseteq S_1\times S_2$ and $V(G^h)\nsubseteq S_1\times S_2$. Then, for any $(g',h')\in S_1\times S_2$, $(g,h)(g',h')\notin E(G\Box H)$. Therefore, $S_1\times S_2$ is not a dominating set in $G\Box H$. Hence $S_1\times S_2$ is not an outer-weakly convex dominating set in $G\Box H$.
%\end{proof}

\begin{figure}[H]
    \centering
\includegraphics[width=90mm,scale=1]{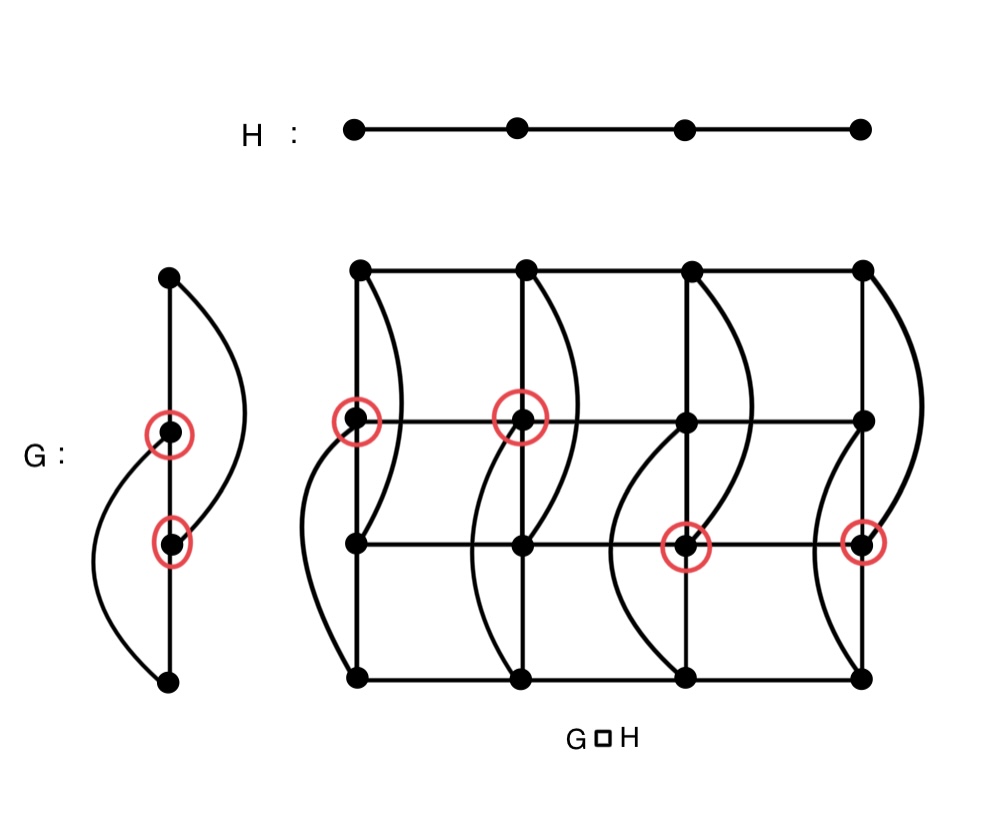}

    \caption{The marked vertices in $G \Box H$ represent an outer-weakly convex dominating set in $G \Box H$. However, their corresponding projection onto $G$ does not form an outer-weakly convex dominating set in $G$.} 
    \label{tol1}
  \centering
   \end{figure}

\begin{theorem}
   Let $G$ and $H$ be two non-trivial connected graphs with order $m$ and $n$, respectively. Then, we have

   $$\min\{m,n\}\leq \widetilde{ \gamma}_{wcon}(G\Box H)\leq \min \{\widetilde{ \gamma}_{wcon}(G)n, \widetilde{\gamma}_{wcon}(H)m\}$$
\end{theorem}
\begin{proof}
To prove the lower bound, suppose for contradiction that there exists an outer-weakly convex dominating set $U$ in $G \Box H$ such that $|U| < \min\{m, n\}$. Then, it follows that $\pi_G(U) \neq V(G)$ and $\pi_H(U) \neq V(H)$. In this case, we can find a pair of vertices $g \in V(G)$ and $h \in V(H)$ such that $V(^gH) \cap U = \emptyset$ and $V(G^h) \cap U = \emptyset$. Then, by the definition of the Cartesian product of graphs, $(g,h) \in V(G \Box H) \setminus U$ and is not adjacent to any vertex in $U$, which contradicts the fact that $U$ is a dominating set in $G \Box H$. Therefore, $|U| \geq \min\{m, n\}$.

For the upper bound, let $S$ and $T$ be minimum outer-weakly convex dominating sets in $G$ and $H$, respectively.\\
{\bf Claim:} $V(G)\times T$ is an outer-weakly convex dominating set in $G\Box H$.\\
First, we prove that $V(G)\times T$ is a dominating set in $G\Box H$. Let $(g,h)$ be any vertex in $V(G\Box H)\setminus (V(G)\times T)$. Since $T$ is a dominating set, there exists at least one vertex, say $h'\in T$ with $hh'\in E(H)$. Then, $(g,h')\in V(G)\times T$ and $(g,h)(g,h')\in E(G\Box H)$. This shows that $V(G)\times T$ is a dominating set in $G\Box H$.

Next, we prove $V(G\Box H)\setminus (V(G)\times T)$ is weakly convex in $G\Box H$. For this, let $(g,h)$ and $(g',h')$ be two vertices in $V(G\Box H)\setminus (V(G)\times T)$. We need at least one shortest path between $(g,h)$ and $(g',h')$ in the induced subgraph of $V(G\Box H)\setminus (V(G)\times T)$. We consider the following three cases.
\\{\bf Case 1}: $g=g'$ and $h\neq h'$. Since both $(g,h)$ and $(g',h')$ are not in $V(G)\times T$, both $h,h'\notin T$. Again, since $V(H)\setminus T$ is weakly convex,  it follows that there exists a $h-h'$ shortest path, say $ h=h_1,h_2,\ldots h_r=h'$ in the induced subgraph of $V(H)\setminus T$. This in turn  implies that the path $(g,h)=(g,h_1),(g,h_2),\ldots, (g,h_r)=(g',h')$ is a $(g,h)-(g',h')$ shortest path in $V(G\Box H)\setminus (V(G)\times T)$.
\\{\bf Case 2}: $g\neq g'$ and $h=h'$. This is similar to Case 1.
\\{\bf Case 3}: $g\neq g'$ and $h\neq h'$. Since $(g,h)$ and $(g',h')$ are not in $V(G)\times T$, both $h,h'\notin T$. Again, since $T$ is an outer-weakly convex set in $H$, $V(H)\setminus T$ is weakly convex in $H$. Then, there exists at least one $h-h'$ shortest path, say $h=h_1,h_2,\ldots h_l=h'$ in the induced subgraph of $V(H)\setminus T$. Let $g=g_1,g_2,\ldots ,g_s=g'$ be any $g-g'$ shortest path in $G$.

 This in turn implies that the path $(g,h)=(g,h_1),(g,h_2),\ldots ,(g,h'), (g_1,h'),(g_2,h'),\\\ldots ,(g_s,h)=(g',h)$ is a  $(g,h)-(g',h')$ shortest path in the induced subgraph of $V(G\Box H)\setminus (V(G)\times T)$.

This proves that $V(G)\times T$ is an outer-weakly convex dominating set in $G\Box H$. Similarly, we can prove that $S\times V(H)$ is an outer-weakly convex dominating set in $G\Box H$. This completes the proof.
\end{proof}

Consider integers $m\geq n\geq 2$. Then, for any nontrivial connected graph $H$ with $\widetilde{\gamma}_{wcon}(H)\geq  \left\lceil \frac{n}{m} \right\rceil $, the above theorem implies that $\widetilde{ \gamma}_{wcon}(K_n\Box H)=n$.

%For the sharpness of the upper bound of the above theorem, consider $G$ as $K_n$ and $P_2$ as $H$. Here $\widetilde{\gamma}_{wcon}(G)=1$ and $\widetilde{\gamma}_{wcon}(H)=1$. Here, no singleton vertices in the product is an outer-weakly convex dominating set. For any $g\in V(G)$, $V(^gH)$ is an outer-weekly convex dominating set in $G\Box H$. Therefore $\widetilde{\gamma}_{wcon}(G\Box H)=\min \{\widetilde{\gamma}_{wcon}(G)|V(H)|, \widetilde{\gamma}_{wcon}(H)|V(G)|\}=\widetilde{\gamma}_{wcon}(G)|V(H)|=1.2=2$.

\section{Strong Product of Graphs}

\begin{theorem}\label{strong}
Let $G$ and $H$ be two non-trivial connected graphs with order $m$ and $n$, respectively. Then, we have

   $$\max\{\gamma(G),\gamma(H)\}\leq \widetilde{ \gamma}_{wcon}(G\boxtimes H)\leq \min \{\widetilde{ \gamma}_{wcon}(G)n, \widetilde{\gamma}_{wcon}(H)m\}.$$   
\end{theorem}
\begin{proof}
   Let $S$ be an outer-weakly convex dominating set in $G\boxtimes H$. First, we prove that $\pi_G(S)$ is a dominating set in $G$. Suppose on the contrary that $\pi_G(S)$ is not a dominating set in $G$. Then, there exists $g\in V(G)\setminus \pi_G(S)$ with $N_G(g)\cap \pi_G(S)=\emptyset$. This implies that any $h\in V(H)$, the vertex $(g,h)$ is not adjacent to any vertex of $\pi_G(S)\times \pi_h(S)$. Hence, we have $N_{G\boxtimes H}(g,h)\cap S=\emptyset$. However, this contradicts the fact that $S$ is an outer-weakly convex dominating set in $G\boxtimes H$. Therefore, $\pi_G(S)$ is a dominating set in $G$. By similar arguments, we conclude that $\pi_H(S)$ is a dominating set of $H$. This shows that every outer-weakly convex dominating set in $G\boxtimes H$ has cardinality at least $\max\{\gamma(G),\gamma(H)\}$.

   For the upper bound, let $S$ be an outer-convex dominating set in $G$. We first prove the following claim.\\
    {\bf Claim:} $S\times V(H)$ is an outer-weakly convex dominating set in $G\boxtimes H$. \\
     Let $(g, h),(g',h')\in V(G\boxtimes H)\setminus (S\times V(H))$. We prove that there exists a shortest path joining $(g,h)$ and $(g',h')$ in the induced subgraph of $V(G\boxtimes H)\setminus (S_1\times S_2)$. If $(g,h)(g',h')\in E(G\boxtimes H)$, then we are done. Otherwise, we consider the following three cases.
\\{\bf Case 1}: $g= g'$ and $hh'\notin E(H)$.
In this case, both the vertices $(g,h)$ and $(g',h')$ are in the same $H$-layer, $^gH$. Recall that $g\notin S$. Let $h=h_1,h_2,\ldots h_r=h'$ be a $h-h'$ shortest path in $H$. Then, the path $(g,h)=(g,h_1),(g,h_2),\ldots, (g,h_r)=(g',h')$ is a $(g,h)-(g',h')$ shortest path in the subgraph induced by $V(G\boxtimes H)\setminus (S\times V(H))$.  
 {\bf Case 2}: $gg'\notin E(G)$ and $h= h'$.   Since $g,g'\notin S$, then there exists a $g-g'$ shortest path $g=g_1,g_2,\ldots,g_r=g'$ in the induced subgraph of $V(G)\setminus S$. Then, as above, the path $(g,h)=(g_1,h),(g_2,h),\ldots, (g_r,h)=(g',h')$ is a $(g,h)-(g',h')$ shortest path in the subgraph induced by $V(G\boxtimes H)\setminus (S\times V(H))$.
\\{\bf Case 3}: $g\neq g'$ and  $h\neq h'$. Since $S$ is an outer-weakly convex set in $G$ and $g,g'\notin S$, then there exists a $g-g'$ shortest path $P:g=g_1,g_2,\ldots,g_r=g'$ in the induced subgraph of $V(G)\setminus S$. Now, consider a $h-h'$ shortest path, say $Q:h=h_1,h_2,\ldots,h_s=h'$ in $H$. Then, $V(P)\times V(Q)\subseteq V(G\boxtimes H)\setminus (S\times V(H))$. If $r=s$, then the path $R:(g,h)=(g_1,h_1),(g_2,h_2),\ldots,(g_r,h_r)=(g',h')$ will be a  $(g,h)-(g',h')$ shortest path in the induced subgraph of $V(G\boxtimes H)\setminus (S\times V(H))$. So, assume without loss of generality that $r<s$. Then, the path 
       $T:(g,h)=(g_1,h_1),(g_2,h_2),\ldots,(g_r,h_r),(g_{r+1},h'),(g_{r+2},h'),\ldots (g',h')$ will be a $(g,h)-(g',h')$ shortest path in the induced subgraph of $V(G\boxtimes H)\setminus (S\times V(H))$.
       
        From all the above cases, we conclude that  $V(G\boxtimes H)\setminus S\times V(H)$ is a weakly convex set in $G\boxtimes H$.

     Next, we prove that $S\times V(H)$ is a dominating set in $G\boxtimes H$. Let $(g,h)\notin S\times V(H)$. Then, $g\notin S$. Since $S$ is an outer-weakly convex dominating set in $G$, there exists a vertex $g'\in S$ such that $gg'\in E(G)$. Now, it is clear that for any $h\in V(H)$, $(g,h)(g',h)\in E(G\boxtimes H)$ and $(g',h)\in S\times V(H)$. Hence, $S\times V(H)$ is a dominating set in $G\boxtimes H$. Hence the claim follows.
    
    As in the above claim, we can prove that for any outer-weakly convex dominating set $T$ of $H$, the set $V(G)\times T$ is an outer-weakly convex dominating set of $G \boxtimes H$. This proves the upper bound.

\end{proof}
The following theorem gives the sharpness of the lower bound in Theorem~\ref{strong}.
\begin{theorem}
    Let $G$ be a non-trivial connected graph and $n$ be any positive integer greater than or equal to 2. Then,  $\widetilde{\gamma}_{wcon}(G\boxtimes K_n)=\gamma(G)$.
\end{theorem}
\begin{proof}

    Let $S$ be a  dominating set in $G$ with $|S|=\gamma(G)$. Fix $h\in V(K_n)$. In the following, we prove that the set $S'=S\times \{h\}$ an outer-weakly convex dominating set in $G\boxtimes K_n$.
First, we prove $S'$ is a dominating set in $G\boxtimes K_n$. Let $(g,h')\in V(G\boxtimes K_n)\setminus S'$. If $g\in S$, then the vertex $(g, h')$ is dominated by the vertex $(g,h)$. So,  assume $g\notin S$.  Since $S$ is a dominating set in $G$, there exists a vertex $g'\in S$ with $gg'\in E(G)$. Then, $(g', h)\in S'$ and $(g,h')(g',h)\in E(G\boxtimes K_n)$. This proves $S'$ is a dominating set in $G\boxtimes H$.

 Next, we prove $S'$ is an outer-weakly convex set in $G\boxtimes K_n$. Let $(g_1,h_1),(g_2,h_2)\in V(G\boxtimes K_n)\setminus S'$. Let $P: g_1=u_1,u_2,\ldots ,u_r=g_2$ be a $g_1-g_2$ shortest path in $G$. Choose $h'$ in $V(K_n)$ so that $h\neq h'$. Then, the path $Q: (g_1,h_1)=(u_1,h_1),(u_2,h'),(u_3,h')\\\ldots,(u_{r-1},h'),(u_r,h_2)=(g_2,h_2)$ is a $(g_1,h_1)-(g_2,h_2)$ shortest path in $G\boxtimes K_n$ with $V(Q)\cap S'=\emptyset$. Therefore, $S'$ is an outer-weakly convex dominating set in $G\boxtimes K_n$. Hence, by Theorem \ref{strong}, $\widetilde{\gamma}_{wcon}(G\boxtimes K_n)=\gamma(G)$. 

\end{proof}

\begin{proposition}\label{Kmn}
   Let $G$ be a non-trivial connected graph. Then, $\widetilde{\gamma}_{wcon}(G\boxtimes K_{m,n})\leq 2{\gamma}(G)$, where $m,n\geq2$.
\end{proposition}

\begin{proof}
    Let $S$ be a dominating set in $G$ with $|S|=\gamma(G)$ and let $hh'\in E(K_{m,n})$. Consider the set $S'=S\times \{h,h'\}$. In the following we prove that $S'$ is an outer-weakly convex dominating set in $G\boxtimes K_{m,n}$. 
To prove that $S'$ is a dominating set in $G \boxtimes K_{m,n}$, consider an arbitrary vertex $(x,y) \in V(G \boxtimes K_{m,n}) \setminus S'$.
Then, we consider the following three cases.   
\\{\bf Case 1}: $x\notin S$ and $y\in \{h,h'\}$. Since $S$ is an outer-weakly convex dominating set in $G$, there exists a vertex $x'\in S$ with $xx'\in E(G)$. Assume for a moment that $y = h$. Then, $(x', h) \in S'$ and $(x, y)(x', h) \in E(G \boxtimes K_{m,n})$.  
Similarly, the case $y = h'$ can be proved.
\\{\bf Case 2}: Let $x \in S$ and $y \notin \{h, h'\}$. Since $hh'\in E(K_{m,n})$, either $yh \in E(K_{m,n})$ or $yh' \in E(K_{m,n})$. Hence, either $(x,y)(x,h) \in E(G \boxtimes K_{m,n})$ or $(x,y)(x,h') \in E(G \boxtimes K_{m,n})$, or both. Since both $(x,h)$ and $(x,h')$ are in $S'$, the vertex $(x,y)$ is dominated by $S'$ in $G \boxtimes K_{m,n}$.
\\{\bf Case 3}: $x\notin S$ and $y\notin \{h,h'\}$.
Since $S$ is a dominating set in $G$, there exists a vertex $x'\in S$ with $xx'\in E(G)$. As in the previous case, either $hy\in E(K_{m,n})$ or $h'y\in E(K_{m,n})$. Assume without loss of generality that $hy\in E(K_{m,n})$. Then, $(x,y)(x',h)\in E(G\boxtimes K_{m,n})$ and $(x',h)\in S'$.

Hence, $S'$ is a dominating set in $G\boxtimes K_{m,n}$.

Next, we prove that $S'$ is an outer-weakly convex set in $G\boxtimes K_{m,n}$. Choose two arbitrary non-adjacent vertices $(x, y)$ and $(x',y')$ in $V(G\boxtimes K_{m,n})\setminus S'$. We consider the following three cases.\\
    {\bf Case 1.} Both $y, y'\in \{h, h'\}$.  This turn implies that $x$ and $x'$ are distinct non-adjacent vertices in $G$ such that $x,x'\notin S$. Since $S$ is an outer-weakly convex set in $G$, there exists an $x- x'$ shortest path in $G$, say $P$, such that $V(P)\cap S=\emptyset$. Fix $P:x=x_0,x_1,\dots,x_k=x'$. Then, the path $Q:(x,y)=(x_0,y),(x_1,y),\dots,(x_{k-1},y)(x_k,y')=(x',y')$ is an $(x,y)-(x',y')$ shortest path in $G\boxtimes K_{m,n}$ with $V(Q)\cap S'=\emptyset$.\\
    {\bf Case 2.}  $y\notin \{h,h'\}$ and $y'\notin \{h,h'\}$. Let $P:x=x_0,x_1,\dots,x_k=x'$ be any $x-x'$ shortest path in $G$. First, consider the case $y=y'$. Then, choose the $(x,y)-(x',y')$ shortest path $Q:(x,y)=(x_0,y), (x_1,y),\dots,(x_k,y)=(x',y')$ in $G\boxtimes K_{m,n}$, and hence it is clear that $V(Q)\cap S'=\emptyset$. Next, consider the case that $y$ and $y'$ are adjacent vertices in $G\boxtimes K_{m,n}$. Then, the path $Q:(x,y)=(x_0,y),(x_1,y,\dots,(x_{k-1},y),(x_k,y')=(x', y')$ is an $(x,y)-(x',y')$ shortest path in $G\boxtimes K_{m,n}$ with $V(Q)\cap S'=\emptyset$. Hence, in the following we assume that the vertices $y$ and $y'$ are distinct non-adjacent vertices in $K_{m,n}$. Since $m\geq n\geq 2$, we can choose a vertex $h_1$ distinct from both $h$ and $h'$ such that $h_1$ is adjacent to both $y$ and $y'$ in $K_{m,n}$. Now, when $k\geq 2$, choose the path $Q:(x,y)=(x_0,y), (x_1,h_1), (x_2,y'), \dots,(x_{k-1},y')=(x',y')$.  On the other hand, when $k=1$, choose $Q:(x,y)=(x_0,y),(x_1,h_1),(x_1,y')$; and when $k=0$, choose $Q:(x,y)=(x_0,y),(x_0,h_1),(x_0,y')=(x',y')$. Then, in all cases, $Q$ is an $(x,y)-(x',y')$ shortest path in $G\boxtimes K_{m,n}$ with $V(Q)\cap S'=\emptyset$.\\
    {\bf Case 3.}  $y\in \{h,h'\}$ and $y'\notin \{h,h'\}$. Assume that $y=h$. Then, $x\notin S$. Let $P:x=x_0,x_1,\dots,x_k=x'$ be any $x-x'$ shortest path in $G$. When $y$ and $y'$ are adjacent in $K_{m,n}$, then  the path $Q:(x,y)=(x_0,h),(x_1, y'), (x_2, y'),\dots,(x_k,y')=(x',y')$ is an $(x,y)-(x',y')$ shortest path in $G\boxtimes K_{m,n}$ with $V(Q)\cap S'=\emptyset$. On the other hand, when $y$ and $y'$ are not adjacent in $K_{m,n}$, then there we can choose a vertex $h_1$ distinct from $h'$ such that $h_1$ is adjacent to both $y$ and $y'$ in $K_{m,n}$. Then, as in Case 2, we can choose a  $(x,y)-(x',y')$ shortest path $Q$ in $G\boxtimes K_{m,n}$ with $V(Q)\cap S'=\emptyset$.

    Thus $S'$ is an outer-weakly convex set of $G\boxtimes K_{m,n}$. This completes the proof.

\end{proof}

If $G=K_r$ (complete graph on $r$ vertices), then any single vertex is not a dominating set in $G\boxtimes K_{m,n}$, and hence, $\widetilde{\gamma}_{wcon}(G\boxtimes K_{m,n})=2{\gamma}(G)=2$. Therefore, the bound in Proposition~\ref{Kmn} is sharp.

\section{Lexicographic Product of Graphs}

As in the case of the Cartesian product, the following proposition shows that the $G$-projection of any outer-convex dominating set of $G \circ H$ is an outer-convex dominating set in $G$. However, the same type of result does not necessarily hold for outer-weakly convex dominating sets, as illustrated in Figure~2.

\begin{proposition}\label{pro:lex1}
    Let $G$ and $H$ be two non-trivial connected graphs. If $S$ is an outer-convex dominating set in $G\circ H$, then $\pi_G(S)$ is an outer-convex dominating set in $G$.
\end{proposition}
\begin{proof}
    Let $S$ be an outer-convex dominating set in $G\circ H$. Suppose on the contrary that $\pi_G(S)$ is not an outer-convex dominating set in $G$. Then, either $\pi_G(S)$ is not a dominating set in $G$, or $V(G)\setminus \pi_G(S)$ is not convex set in graph $G$, or both. 
    
First, assume $\pi_G(S)$ is not a dominating set in $G$. Then, there exists a vertex $g\in V(G)\setminus \pi_G(S)$ with $N_G(g)\cap \pi_G(S)=\emptyset$. Then, for any $h\in V(H)$, the vertex $(g,h)\in V(G\circ H)\setminus S$ and $N_{G\circ H}(g,h)\cap S=\emptyset$. This contradicts the fact that $S$ is an outer convex dominating set in $G\circ H$. Therefore, $\pi_G(S)$ is a dominating set in $G$. 

Next, assume that $V(G)\setminus \pi_G(S)$ is not convex in $G$. Then, there exist two vertices $g',g''\in V(G)\setminus \pi_G(S)$ and a $g'-g''$ shortest path in $G$ that intersects $\pi_G(S)$. We fix $P:g'=g_1,g_2,\ldots, g_r=g''$   with $g_i\in \pi_G(S)$ for some $i\in [r]$. Now, choose $h\in V(H)$ such that $(g_i, h)\in S$.  Then both $(g',h)$ and $(g'',h)$ are not in $S$. By the definition of lexicographic product of graphs, the path $Q: (g',h)=(g_1,h),(g_2,h), \ldots,(g_k,h)= (g'',h)$ is a $(g',h)-(g'',h)$ shortest path in $G\circ H$ containing the vertex $(g_i, h)$, a contradiction to the fact that $S$ is  outer-convex in $G\circ H$. Hence $\pi_G(S)$ is an outer-convex dominating set in $G$. 
\end{proof}

\begin{figure}[H]
    \centering
\includegraphics[width=65mm,scale=.4]{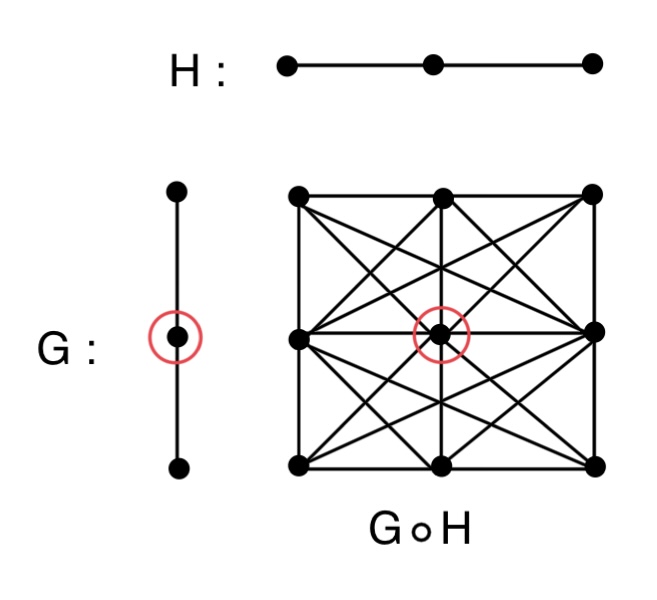}

    \caption{The marked vertex in $G \circ H$ represent an outer-weakly convex dominating set in $G \circ H$. However, their corresponding projection onto $G$ does not form an outer-weakly convex dominating set in $G$.} 
    \label{tol1}
  \centering
   \end{figure}

In the following, we determine the exact values of the outer-weakly convex domination number of lexicographic product graphs in terms of the domination number and the total domination number of the factor graphs. For a graph $G$ with no isolated vertices, 
a set $S \subseteq V(G)$ is called a \emph{total dominating set} of $G$ if every vertex of $G$ has at least one neighbor in $S$. 
The \emph{total domination number} of $G$, denoted by $\gamma_t(G)$, is the minimum cardinality of a total dominating set of $G$. The total domination number has been extensively studied, as detailed in the book \cite{rtotal}.

\begin{theorem}\label{lexico}
    Let $G$ a non-trivial connected graphs. Then for any non-trivial graph $H$, we have
     $$\widetilde{\gamma}_{wcon}(G\circ H) =
\begin{cases}
\gamma(G), & \text{if } \gamma(H) =1\\
\gamma_t(G), & \text{otherwise } 
\end{cases}$$
\end{theorem}
\begin{proof} Since the projection of any dominating set of $G\circ H$ onto $G$ itself is a dominating set of $G$, it follows that $\widetilde{\gamma}_{wcon}(G\circ H)\geq \gamma(G)$.  

   On the other hand, let $S$ be a minimum total dominating set of $G$.
In the following, we prove that the set $S' = S \times \{h\}$ is an outer-weakly convex dominating of in $G \circ H$, for any $h \in V(H)$.

We first prove that $S'$ is a dominating set of $G\circ H$. Let $(g',h')\in V(G\circ H)\setminus S'$. Since $S$ is a total dominating set of $G$, there is a vertex $g''\in S$ with $g'g''\in E(G)$.  This in turn implies that the vertex $(g',h')$ is dominated by $(g'',h)\in S'$.

Next, we prove that $V(G\circ H)\setminus S'$ is a weakly convex set of $G\circ H$.  Let $(x,y),(x',y')\in V(G\circ H)\setminus S'$. In the following, we construct an $(x,y)-(x',y')$ shortest path in the induced subgraph of $V(G\circ H)\setminus S'$. We consider the following two cases.\\
{\bf Case 1}:  $x=x'$. If $y$ and $y'$ are adjacent in $H$, then it is obvious.  So, assume $yy'\notin E(H)$. Then either $y\neq h$ or $y'\neq h$, say $y\neq h$.  Since $G$ is a non-trivial connected graph, there exists $x''\in V(G)$ such that $xx''\in E(G)$. Then the  path  $Q: (x,y), (x'', y ), (x',y)$ is a $(x,y)-(x',y')$ shortest path in the induced subgraph of $V(G\circ H)\setminus S'$.\\
{\bf Case 2}: $x\neq x'$.  Let $P: x=x_1, x_2,\ldots,x_r=x'$ be a $x-x'$  shortest path in $G$. Then,  for any $h'\neq h$,  the path $Q: (x,y)=(x_1,y),(x_2,h'),\ldots, (x_{r-1}, h'), (x_r,y)=(x',y')$ is a shortest $(x,y),(x',y')$-path in the induced subgraph of  $V(G\circ H)\setminus S'$.

Hence,  $S'$ is an outer-weakly convex set of $G\circ H$. Therefore, $\widetilde{\gamma}_{wcon}(G\circ H)\leq |S'| \leq {\gamma}_t(G)$.

Now, suppose that $\gamma(H)=1$, and let $\{h\}$ be a dominating set of $H$.  Then, for any minimum dominating set $S$ of $G$,  using arguments similar to the above, we can prove that the set $S\times \{h\}$ is an outer-weakly convex dominating set of $G\circ H$. Consequently,  $\widetilde{\gamma}_{wcon}(G\circ H) =\gamma(G)$.

On the other hand, assume that $\gamma(H)\geq 2$. Let $S$ be a minimum outer-weakly convex dominating set of $G\circ H$, and recall that $\pi_G(S)$ forms a dominating set of $G$. Now, let $g$ be an isolated vertex in the induced subgraph of $\pi_G(S)$ in $G$. Since $S$ is a dominating set of $G\circ H$ and $H$ has no universal vertices, it follows that $|S\cap (\{g\}\times V(H)|\geq 2$. Hence, each isolated vertex $g$ in the induced subgraph of $\pi_G(S)$ contributes at least two vertices to the set $S$ from the corresponding $H$ layer $^gH$ of $G\circ H$.  This in turn implies that $S$ contains at least $\gamma_t(G)$ vertices and so $\widetilde{\gamma}_{wcon}(G\circ H) =\gamma_t(G)$.
\end{proof}
Applying the known results for the domination and total domination numbers of complete graphs, paths, cycles, and complete bipartite graphs presented in \cite{r5,rtotal}, we obtain the following closed formulas for the outer-weakly convex domination numbers of the lexicographic product of certain classes of graphs.
\begin{corollary}
    Let $H$ be a non-trivial graph with $\gamma (H)= 1$. Then for positive integers $m,n\geq 2$, the following holds.

    \begin{enumerate}
      \item $\widetilde{\gamma}_{wcon}(K_n\circ H) =\gamma(K_n)=1$

\item $\widetilde{\gamma}_{wcon}(K_{m,n}\circ H) =\gamma(K_{m,n})=2$

\item $\widetilde{\gamma}_{wcon}(P_n\circ H) =\gamma(P_n)=\left\lceil \frac{n}{3} \right\rceil$

 \item $\widetilde{\gamma}_{wcon}(C_n\circ H) =\gamma(C_n)=\left\lceil \frac{n}{3}  \right\rceil \text{for}$ $ n\geq 3$

    \end{enumerate}
\end{corollary}
\begin{corollary}
    Let $H$ be a non-trivial graph with $\gamma (H)\geq 2$. Then for positive integers $m\geq 2$ and $n\geq 3$, the following holds.

    \begin{enumerate}
      \item $\widetilde{\gamma}_{wcon}(K_n\circ H) =\gamma_t(K_n)=2$

\item $\widetilde{\gamma}_{wcon}(K_{m,n}\circ H) =\gamma_t(K_{m,n})=2$

\item $\widetilde{\gamma}_{wcon}(P_n\circ H) =\widetilde{\gamma}_{wcon}(C_n\circ H) =\gamma_t(C_n)=\gamma_t(P_n)=\begin{cases}
\tfrac{n}{2} , & \text{if } n \equiv 0 \pmod{4}, \\[8pt]
 \tfrac{n+1}{2}, & \text{if } n \equiv 1,3 \pmod{4}, \\[8pt]
 \tfrac{n}{2}+1, & \text{if } n \equiv 2  \pmod{4}.
\end{cases}$
    \end{enumerate}
\end{corollary}

\section{Conclusion} The main contribution of this paper is the determination of tight bounds for the outer-weakly convex domination number on Cartesian and strong product graphs. Furthermore, the exact values of the outer-weakly convex domination number for lexicographic product of graphs are determined in terms of the domination and total domination numbers of one of the factor graphs. Our results suggest that finding the exact value for the outer weakly-convex domination number is a challenging task for the Cartesian and strong products. However, it would be of interest to continue the study on restricted classes of product graphs.
\section*{Acknowledgments} The authors are grateful to the three anonymous reviewers for their constructive feedback, which helped enhance the clarity and rigor of this work.

\end{document}